\documentclass[12pt]{amsart}

\calclayout

\usepackage[lite]{amsrefs}

\usepackage{amssymb}

\usepackage[all,cmtip]{xy}

\usepackage{mathrsfs}

\usepackage{tabularx}
\usepackage{booktabs}
\usepackage[labelfont=bf,format=plain,justification=raggedright,singlelinecheck=false]{caption}

\newcommand{\F}{\mathbb{F}}

\numberwithin{equation}{section}

\theoremstyle{plain} 
\newtheorem{thm}[equation]{Theorem}

\newtheorem{prop}[equation]{Proposition}

\theoremstyle{definition}

\theoremstyle{remark}
\newtheorem{rem}[equation]{Remark}

\usepackage{hyperref}

\begin{document}

\title[On the constant $\protect D(q)$ defined by Homma]{On the constant $\protect D(q)$ defined by Homma}

\author{Peter Beelen} \address{Department of Applied Mathematics and Computer Science, Technical University of Denmark, Kongens Lyngby 2800, Denmark} \email{pabe@dtu.dk} \thanks{}
\author{Maria Montanucci} \address{Department of Applied Mathematics and Computer Science, Technical University of Denmark, Kongens Lyngby 2800, Denmark} \email{marimo@dtu.dk} \thanks{}
\author{Lara Vicino} \address{Department of Applied Mathematics and Computer Science, Technical University of Denmark, Kongens Lyngby 2800, Denmark}  \email{lavi@dtu.dk} \thanks{}

\address{}

\date{}

 \begin{abstract}
Let $\mathcal{X}$ be a projective, irreducible, nonsingular algebraic curve over the finite field $\mathbb{F}_q$ with $q$ elements and let $|\mathcal{X}(\mathbb{F}_q)|$ and $g(\mathcal X)$ be its number of rational points and genus respectively. The Ihara constant $A(q)$ has been intensively studied during the last decades, and it is defined as the limit superior of $|\mathcal{X}(\mathbb{F}_q)|/g(\mathcal X)$ as the genus of $\mathcal X$ goes to infinity. In \cite{homma2012bound} an analogue $D(q)$ of $A(q)$ is defined, where the nonsingularity of $\mathcal X$ is dropped and $g(\mathcal X)$ is replaced with the degree of $\mathcal X$. We will call $D(q)$ Homma's constant. In this paper, upper and lower bounds for the value of $D(q)$ are found.
\end{abstract}

\maketitle

\vspace{0.5cm}\noindent {\em Keywords}:
Algebraic curve, rational point, finite field

\vspace{0.2cm}\noindent{\em MSC}:
14G15, 14H50, 11G20, 14H25
\vspace{0.2cm}\noindent

\section{Introduction}

Let $p$ be a prime and let $q = p^e$ be a prime power. Let $\mathcal{X}$ be a projective, nonsingular, geometrically irreducible curve of genus $g$. The interaction between the genus $g$ of $\mathcal{X}$ and the number $|\mathcal{X}(\mathbb{F}_q)| $ of its rational points has been subject of intense studies during the last years. It is well known that the Weil bound
$$|\mathcal{X}(\mathbb{F}_q)| \leq q + 1 + 2g\sqrt{q}$$
is not sharp if $g$ is large compared to $q$.
Put
\begin{equation} \label{max}
N_q(g) := \max|\mathcal{X}(\mathbb{F}_q)|,
\end{equation}
where the maximum is taken over all curves $\mathcal{X}/\mathbb{F}_q$ with genus $g$. The \textit{Ihara constant} is defined by
\begin{equation} \label{IharaC}
A(q) := \limsup_{g\rightarrow \infty} \frac{N_q(g)}{g}.
\end{equation}
This is a measure of the asymptotic behaviour of the number of rational points on curves over $\mathbb{F}_q$ when the genus becomes large.
Ihara's constant $A(q)$ has been intensively studied during the last decades. For any $q$, we have $A(q) \leq \sqrt{q}-1$ (see \cite{DV}), and if $q$ is a square we have (see \cite{Ihara,TVZ}) $A(q) = \sqrt{q}-1$.

For any $q$, using class field theory, Serre \cite{Serre} showed that $A(q)> c \log(q)$ for some constant $c >0$ independent of $q$. In particular $A(q)>0$ for all $q$. For $q=p^{2m+1}$, with $m>0$, the currently best-known lower bound is $A(q) \ge 2(1/(p^m-1)+1/(p^{m+1}-1))^{-1}$, see \cite{BBGS}. The exact value of $A(q)$ is however unknown when $q$ is not a square.

If the curve $\mathcal{X}$ is seen as a projective curve $\mathcal{X} \subseteq \mathbb{P}^n(\mathbb{F}_q)$ of degree $d>0$ and it is not necessarily required to be nonsingular, a different question can be addressed: how large can $|\mathcal{X}(\mathbb{F}_q)|$ be with respect to $d$?

In a series of papers \cite{H1,H2,H3} it has been shown that if $\mathcal{X}$ is a (possibly reducible) plane curve without $\mathbb{F}_q$-linear components, then
\begin{equation} \label{Sziklai}
|\mathcal{X}(\mathbb{F}_q)| \leq  (d - 1)q + 1,
\end{equation}
except for curves isomorphic over $\mathbb{F}_4$ to the curve defined by
$$\mathcal{K} : (X + Y + Z)^4 + (XY + Y Z + ZX)^2 + XYZ(X + Y + Z) = 0,$$
which satisfies $|\mathcal{K}(\mathbb{F}_4)| = 14$. The bound \eqref{Sziklai} was originally conjectured by Sziklai \cite{SZI}, and he found that some curves actually achieve this bound.

The natural question on whether the bound \eqref{Sziklai} is valid for curves in higher dimensional projective space $n \geq 3$ was analyzed by Homma in \cite{homma2012bound}. There, it is obtained that \eqref{Sziklai} is also true when $n \geq 3$ and $\mathcal{X}$ has no $\mathbb{F}_q$-linear components, unless $d=q=4$ and $\mathcal{X}$ is $\mathbb{F}_q$-isomorphic to the plane curve $\mathcal{K}$.

In the same paper \cite{homma2012bound}, an analogue of Ihara constant $A(q)$ \eqref{IharaC} is given when replacing the genus $g$ with the degree $d$. First, we replace $N_q(g)$ as defined in \eqref{max}, with $M_q(d):=\max |\mathcal{X}(\mathbb{F}_q)|$ where this time the maximum is taken over all irreducible curves of a fixed degree $d$ in a projective space of some dimension over $\mathbb{F}_q$. Here the dimension is not fixed and therefore allowed to be arbitrarily large. Then the analogue of $A(q)$ is defined as
\begin{equation} \label{Dq}
D(q) := \limsup_{d \rightarrow \infty}\frac{M_q(d)}{d},
\end{equation}
which measures the asymptotic behavior of the number of rational points of projective curves over $\mathbb{F}_q$ when $d$ becomes large. In \cite{homma2012bound} it was observed that since the bound \eqref{Sziklai} is valid for curves in any projective space $\mathbb{P}^n(\F_q)$, $n \geq 2$, with the exception already mentioned above, one may conclude that $D(q) \le q$. In the same paper also the lower bound $D(q) \ge A(q)/2$ was derived, but the exact value of $D(q)$ remains unknown for all $q$.

In this paper, new upper and lower bounds for the value of $D(q)$, which we from now on will call Homma's constant, are found by a refinement of Homma's methods and by using towers of algebraic function fields.
Our main results are summarized in the following theorem.
\begin{thm} \label{main}
Let $q=p^e$ be a prime power and let $D(q)$ be Homma's constant as defined in \eqref{Dq}. Then
\begin{enumerate}
\item $D(q) \leq q-1$,
\item $D(q) \geq 1$ provided that $q > 2$,
\item $D(q^2) \geq \frac{q}{q+1}A(q^2) = \frac{q^2-q}{q+1}$.
\end{enumerate}
\end{thm}
Note that the lower bound $D(q) \ge 1$ is interesting for small values of $q$ only, since otherwise Homma's lower bound $D(q) \ge A(q)/2$ is better. The values $q \le 31$ for which the lower bound $D(q) \ge 1$ is currently the best known are listed in Remark 4.6.

The paper is organized as follows. We start by slightly improving Homma's upper bound on $D(q)$ in Section \ref{UB} by refining his argument, thus proving Item 1 of Theorem \ref{main}. Next we prove Item 2 of Theorem \ref{main} in Section \ref{LB} by explicitly constructing a sequence of curves whose degrees are close to their number of rational points. Finally, the main part of the paper is devoted to proving Item 3 of Theorem \ref{main} in the final section.

\section{An upper bound for $D(q)$: the proof of Item 1 in Theorem \ref{main}} \label{UB}

The upper bound $D(q) \leq q$ obtained by Homma in \cite{homma2012bound}*{Proposition 5.4} was deduced from the bound \eqref{Sziklai}, but in the same paper the following theorem was given.
\begin{thm}[\cite{homma2012bound}*{Theorem 3.2}]
Let $\mathcal{X}$ be a nondegenerate irreducible curve of degree $d$ in $\mathbb{P}^n(\F_q)$. Then
\begin{equation}
\label{eq:hommaub}
|\mathcal{X}(\F_q)| \leq \frac{(q-1)(q^{n+1}-1)}{q(q^n-1)-n(q-1)}d.
\end{equation}
\end{thm}
Here the word \emph{nondegenerate} means that $\mathcal{X}$ is not contained in any hyperplane of $\mathbb{P}^n(\F_q)$. At this point, using this result, we are ready to prove Item 1 in Theorem \ref{main}.

Indeed for a fixed value of $q$, considering equation \eqref{eq:hommaub} and dividing both sides by $d$ gives
\begin{equation}\label{eq:hommaub2}
\frac{|\mathcal{X}(\F_q)|}{d}\leq \frac{(q-1)(q^{n+1}-1)}{q(q^n-1)-n(q-1)}=\dfrac{(q-1)\dfrac{(q^{n+1}-1)}{q^{n+1}}}{\dfrac{q(q^n-1)}{q^{n+1}}-\dfrac{n(q-1)}{q^{n+1}}}.
\end{equation}
This observation can be used to improve the upper bound for $D(q)$. Note that by taking the $\limsup_{d\rightarrow \infty}{M}_q(d)/d$ as in \eqref{Dq}, we are by definition of $D(q)$ considering curves of increasing degree. However, the dimension of the projective spaces containing the curves will be increasing as $d$ increases. Indeed, if for a family of curves $(\mathcal{X}_i)_{i \ge 0}$, with degrees $d_i$ tending to infinity as $i$ tends to infinity, there exist an $n$ such that for all $i$, $\mathcal{X}_i \subseteq \mathbb{P}^n$, then $|\mathcal{X}_i(\F_q)| \le |\mathbb{P}^n(\F_q)| = (q^{n+1}-1)/(q-1)$, implying that $|\mathcal{X}_i(\F_q)|/d_i$ tends to zero as $i$ tends to infinity.

Now let $(\mathcal{X}_i)_{i \ge 0}$, be a family of curves with degrees $d_i$ tending to infinity such that $\limsup_{i \to \infty} |\mathcal{X}_i(\F_q)|/d_i >0$. Further assume for each $i$ that $\mathcal{X}_i$ is a nondegenerate curve contained in $\mathbb{P}^{n_i}$. We have seen that $n_i$ tends to infinity as $i$ tends to infinity. But then we obtain from equation \eqref{eq:hommaub2}:
\begin{equation*}
D(q) \leq \lim_{i\rightarrow\infty}\dfrac{(q-1)\dfrac{(q^{n_ i+1}-1)}{q^{n_i+1}}}{\dfrac{q(q^{n_i}-1)}{q^{n_i+1}}-\dfrac{n_i(q-1)}{q^{n_i+1}}}=q-1.
\end{equation*}

This proves Item 1 of Theorem \ref{main}.

\section{A lower bound for $D(q)$: the proof of Item 2 in Theorem \ref{main}} \label{LB}

For a prime power $q=p^e$ strictly larger than two, consider the tower of function fields $\mathcal{T}=(T_m)_{m\geq 1}$ over $\F_q$ defined recursively as
\[
T_1=\F_{q}(x_1) \qquad \mbox{and} \qquad T_{i+1}=T_i(x_{i+1}) \qquad \mbox{with} \qquad x_{i+1}^{q-1}=-1+(x_i+1)^{q-1}.
\]
The tower $\mathcal{T}$ is similar to an asymptotically good tower considered in \cite{stichtenoth2009algebraic}*{Proposition 7.3.3}, but the variation we consider is actually not asymptotically good. It is not hard to see that the place of $T_1$ corresponding to the zero of $x_1$ is totally ramified in the tower. In particular, the equation $x_{i+1}^{q-1}=-1+(x_i+1)^{q-1}$ is absolutely irreducible when viewed as a polynomial in $T_{i}[x_{i+1}].$ This implies in particular that the ideal $I_\ell:=\langle x_{2}^{q-1}+1-(x_1+1)^{q-1},\dots,x_{\ell}^{q-1}+1-(x_{\ell-1}+1)^{q-1}\rangle \subseteq \F_q[x_1,\dots,x_\ell]$ is a prime ideal. Since we want to deal with projective curves, the following proposition is essential.

\begin{prop}\label{prop:homprime}
Let $\ell>1$ be an integer and define $I'_\ell:=\langle x_{2}^{q-1}+z^{q-1}-(x_1+z)^{q-1},\dots,x_{\ell}^{q-1}+z^{q-1}-(x_{\ell-1}+z)^{q-1}\rangle \subseteq \F_q[x_1,\dots,x_\ell,z].$ Then $I'_\ell$ is a homogeneous prime ideal and the homogenization of the prime ideal $I_\ell:=\langle x_{2}^{q-1}+1-(x_1+1)^{q-1},\dots,x_{\ell}^{q-1}+1-(x_{\ell-1}+1)^{q-1}\rangle \subseteq \F_q[x_1,\dots,x_\ell]$.
\end{prop}
\begin{proof}
For convenience, let us write $g_i:=x_{i+1}^{q-1}+1-(x_i+1)^{q-1}$ and $g_i':=x_{i+1}^{q-1}+z^{q-1}-(x_i+z)^{q-1}.$ We have already seen that the ideal $I_\ell$ is a prime ideal. Now let $>_\mathrm{deglex}$ denote the degree-lexicographic ordering with $x_\ell>_\mathrm{deglex} \ldots >_\mathrm{deglex} x_1 $ as a monomial order in $\F_q[x_1,\ldots,x_\ell]$. Since under this monomial ordering the leading terms of the $g_i$ are co-prime, the set $\{g_1,\dots,g_{\ell-1}\}$ is a Gr\"obner basis of $I_\ell.$
Then from \cite{CLO}*{$\S$ 8.4, Theorem 4} $\{g_1',\ldots, g_{\ell-1}'\}$ is a Gr\"obner basis for the homogenization of $I_\ell$. Hence $I_\ell'$ is the homogenization of the prime ideal $I_\ell$ and in particular $I_\ell'$ is a homogeneous prime ideal.
\end{proof}

Now consider the projective curve $\mathcal{X}_\ell \subset \mathbb{P}^{\ell}$ defined over $\F_q$ given by the homogeneous equations
\begin{equation}\label{eq:DGge1}
x_{i+1}^{q-1}=-z^{q-1}+(x_i+z)^{q-1} \quad \text{for $i=1,\ldots,\ell-1$.}
\end{equation}
Proposition \ref{prop:homprime} implies that $\mathcal{X}_\ell \subset \mathbb{P}^\ell$ is indeed an irreducible projective curve. It actually implies that $\mathcal{X}_\ell$ is a complete intersection, which in turn implies that $\deg(\mathcal{X}_\ell)=\deg(g_1')\cdots \deg(g_{\ell-1}')=(q-1)^{\ell-1}.$

Now we consider the number of $\F_q$-rational points on $\mathcal{X}_\ell$. To estimate this number, we consider the number of projective points $[x_1:x_2:\cdots :x_\ell:0]$ satisfying equation \eqref{eq:DGge1}. Substituting $z=0$ in equation \eqref{eq:DGge1}, we obtain that
\[
x_{i+1}^{q-1}=x_{i}^{q-1} \quad \text{for $i=1,\ldots,\ell-1$.}
\]
Choosing $x_1=1$, we see that any solution is defined over $\F_q$ and that there are exactly $(q-1)^{\ell-1}$ points at the infinity on $\mathcal{X}_{\ell}$. In particular, $|\mathcal{X}_\ell(\F_q)| \ge (q-1)^{\ell-1}.$ Hence
\begin{equation*}
D(q)\geq \limsup_{\ell\rightarrow \infty}\frac{|\mathcal{X}_{\ell}(\F_{q})|}{\deg(\mathcal{X}_{\ell})}\geq \frac{(q-1)^{\ell-1}}{(q-1)^{\ell-1}}=1.
\end{equation*}
This completes the proof of Item 2 of Theorem \ref{main}.

\section{A lower bound for $D(q^2)$: the proof of Item 3 in Theorem \ref{main}} \label{Secq2}

In order to prove Item 3 in  Theorem \ref{main} we use a tower of function fields over $\F_{q^2}$ constructed recursively by Garcia and Stichtenoth in \cite{garcia1996asymptotic} as follows:
\[
F_1=\F_{q^2}(x_1) \qquad \mbox{and} \qquad F_{i+1}=F_i(x_{i+1}) \qquad \mbox{with} \qquad x_{i+1}^q+x_{i+1}=\frac{x_i^q}{x_i^{q-1}+1}.
\]
This tower is optimal in the sense that if $N_1(F_i)$ denotes the number of rational places and $g(F_i)$ the genus of $F_i$, then $\lim_ {m \rightarrow \infty} N_1(F_m)/g(F_m)=q-1=A(q^2)$.

Indeed, any zero of the function $x_1-\alpha$ in $F_1$ for $\alpha \in \F_{q^2} \setminus \{\alpha \mid \alpha^q+\alpha=0\}$ splits completely in the extension $F_m/F_1$, implying that $N_1(F_m) \ge (q-1)q^m$. Moreover, in \cite[Remark 3.8]{garcia1996asymptotic}, the genus $g(F_m)$ of $F_m$ is computed for all $m \geq 1$. It is given by
\[
g(F_m)=\begin{cases}
(q^{m/2}-1)^2 & \mbox{if} \quad m\equiv 0 \pmod 2, \\
(q^{\frac{m+1}{2}}-1)(q^{\frac{m-1}{2}}-1) & \mbox{if} \quad m \equiv 1 \pmod 2.
\end{cases}
\]
Hence optimality of the tower follows.
For computing the genus $g(F_m)$, it is proven that the pole $P_{\infty}$ of $x_1\in F_1$ is totally ramified in all extensions $F_m/F_1$, $m\geq 2$, see also \cite[Proposition 1.1]{pellikaan1998appeared}. We denote by $P_{\infty}^{(m)}$ the unique extension of $P_{\infty}$ in $F_m$. Note that $P_{\infty}^{(m)}$ is a rational place, since  $P_{\infty}$ is totally ramified in $F_m/F_1$.

Even though it is in general a difficult challenge to compute the Weierstrass semigroups at places in a tower, Pellikaan, Stichtenoth, and Torres \cite{pellikaan1998appeared} computed the Weierstrass semigroup at $P_{\infty}^{(m)}$ for all $m \geq 1$. The nice property proven by the authors in \cite{pellikaan1998appeared} is that the semigroups at $P_{\infty}^{(m)}$ can be computed from the one at $P_{\infty}^{(m-1)}$, following a recursive procedure. Indeed from \cite[Theorem 3.1]{pellikaan1998appeared}
\begin{equation}
\label{eq:Wsemigroup}
H(P_{\infty}^{(m)})=\begin{cases}
\mathbb{Z}_{\geq 0} & \mbox{if} \quad m=1\\
qH(P_{\infty}^{(m-1)}) \cup \mathbb{Z}_{\geq c_m} & \mbox{if} \quad m>1
\end{cases}
\end{equation}
where $c_m:=q^m-q^{\lceil \frac{m}{2}\rceil}$ is the conductor of $H(P_{\infty}^{(m)})$.

Let $\{\gamma_1,\ldots,\gamma_{\ell}\}$ be a set of generators of $H(P_{\infty}^{(m)})$, so that
\begin{equation*} 
H(P_{\infty}^{(m)})=\langle \gamma_1,\ldots,\gamma_{\ell}\rangle,
\end{equation*}
and $0<\gamma_1<\cdots <\gamma_{\ell}$. 
Note that equation \eqref{eq:Wsemigroup} implies that $\gamma_1=q^{m-1}$, being the smallest positive element of $H(P_\infty^{(m)})$. This implies that $H(P_\infty^{(m)}) \cap \mathbb{Z}_{< c_m+q^{m-1}}$ is a generating set and that therefore we may assume that
\begin{equation}\label{eq:sizegammaell}
\gamma_\ell \le c_m+q^{m-1}-1.
\end{equation}

By definition of the Weierstrass semigroup $H(P_{\infty}^{(m)})$, there exist functions $f_1,\ldots,f_\ell \in F_m$ such that
$$(f_i)_{\infty}=\gamma_iP_{\infty}^{(m)}, \ i=1,\ldots,\ell.$$
In \cite{saints_heegard_1995}, the functions $f_1,\dots,f_\ell$ are used to define a birational morphism between a nonsingular projective curve $\mathcal{X}$ and a curve $\mathcal{X}'$, with only one point at infinity. Since we use the language of function fields, we need to reformulate the results from \cite{saints_heegard_1995} slightly. Intuitively, we simply use the functions $f_1,\dots,f_n$ to define a map from the set of places of $F_m$ to an algebraic curve $\mathcal{X}_m$. However, this map, which we denote by $\varphi_{m}$, is easiest to describe when first extending the constant field of $F_m$ to $\overline{\F}_q$, the algebraic closure of $ \F_q$, since then all places are rational:
\[
\varphi_{m}: \mathbb{P}(\overline{\F}_qF_m)  \longrightarrow  \mathbb{P}^\ell
\]
defined by
\[
\begin{array}{rcll}
\varphi_{m}(Q) & = &  [1:f_1(Q):\cdots :f_{\ell}(Q)], & \text{ if } \ Q\neq P_\infty^{(m)},\\
\varphi_{m}(Q) & = & [0:\dots:0:1], & \text{ otherwise.} 
\end{array}
\]
Note that \cite{goldschmidt2006algebraic}*{Theorem 4.2.2} implies that indeed the image of the map $\varphi_m$ is a projective curve $\mathcal{X}_m$.
Since $f_1,\dots,f_\ell$ are defined over $\F_{q^2}$, so is $\mathcal{X}_m$. Therefore we will from now on consider the curve $\mathcal{X}_m$ as a curve defined over $\F_{q^2}$.
Moreover, \cite[Theorem 15]{saints_heegard_1995} states among other things that the function field of $\mathcal{X}_m$, when considered over the field $\F_{q^2}$, is exactly $F_m$, that apart from possibly $\varphi_m(P_\infty^{(m)})$, the curve has no singularities and that $P_\infty^{(m)}$ is the only place of $F_m$ centered at $\varphi_m(P_\infty^{(m)})$.
In particular $\varphi_m$ induces a bijection between $\mathbb{P}(\overline{\F}_qF_m) \setminus \{P_\infty^{(m)}\}$ and $\mathcal{X}_m \setminus \{\varphi_m(P_\infty^{(m)})\}$.

\begin{rem}
\label{rem:nondegenerate}
The curve $\mathcal{X}_m$ is a non-degenerate curve in $\mathbb{P}^{\ell}$. Indeed if this was not the case, then there would exist a combination
$a_1+a_2f_1+\cdots + a_{\ell +1}f_{\ell}$, for some $a_i\in \overline{\F}_q$ not all equal to zero, such that $a_1+a_2f_1+\cdots + a_{\ell +1}f_{\ell}\equiv 0$, which is impossible by the linear independence of $\{1,f_1,\ldots,f_{\ell}\}$ over $\overline{\F}_q$ given by \cite{stichtenoth2009algebraic}*{Proposition 3.6.1}.
\end{rem}

Now we investigate the degree and number of $\F_{q^2}$-rational points on $\mathcal{X}_m$. The number of rational points is easy to bound, since the rational places of $F_m$ are in bijection with the points on $\mathcal{X}_m$ defined over $\F_{q^2}$. Indeed, the place $P_\infty^{(m)}$ corresponds to the projective point $[0:\dots:0:1]$, while the remaining rational points of $\mathcal{X}_m$ are non-singular and hence each corresponds to a unique rational place of $F_m$. This shows that 
\begin{equation}\label{eq:ptsXm}
|\mathcal{X}_m(\F_{q^2})| =N_1(F_m) \ge (q-1)q^m.
\end{equation}
The inequality $N_1(F_m) \ge (q-1)q^m$ was already mentioned before.

At this point we need to derive some information on the degree $\deg(\mathcal{X}_m)$ of the curve $\mathcal{X}_m$. The following inequality holds:
\begin{equation}\label{eq:degXm}
\deg(\mathcal{X}_m)\leq \gamma_{\ell} \leq c_m+q^{m-1}-1.
\end{equation}
This can be proven as follows. First of all, the last inequality is simply equation \eqref{eq:sizegammaell}. Now recall that the degree can also be seen as the the maximum number of intersection points with a hyperplane. The points of intersection of the curve $\mathcal{X}_m$ and a hyperplane of equation $a_0x_0+\cdots+ a_{\ell}x_{\ell}=0$ in $\mathbb{P}^{\ell}$ correspond, by the definition of $\varphi_m$, to the places that are zeros of the function $\sum_{i=0}^{\ell}a_i f_i\in \mathscr{L}(\gamma_{\ell}\bar{P}_{\infty}^{(m)})$. Here $\mathscr{L}(\gamma_{\ell}\bar{P}_{\infty}^{(m)})$ denotes the Riemann--Roch space of the divisor $\gamma_{\ell}\bar{P}_{\infty}^{(m)}$. Since the pole divisor of $\sum_{i=0}^{\ell}a_i f_i$ has degree at most $\gamma_\ell$ the same is true for its zero divisor. Hence the number of intersection points is at most $\gamma_{\ell}$.

Now combining equations \eqref{eq:ptsXm} and \eqref{eq:degXm}, we obtain:
\begin{equation*}
D(q^2)\geq \limsup_{m\rightarrow \infty}\frac{|\mathcal{X}_m(\F_{q^2})|}{\mathrm{deg}(\mathcal{X}_m)} \ge \limsup_{m \rightarrow \infty} \frac{(q-1)q^m}{c_m+q^{m-1}-1}=\frac{q^2-q}{q+1}.
\end{equation*}
Since $A(q^2)=q-1$, Item 3 of Theorem \ref{main} follows. 

\begin{rem}

Thereom \ref{main} (3) improves Homma's lower bound $D(q^2) \geq A(q^2)/2$ for any values of $q$. 
The bound $D(q)\ge 1$ is instead interesting for small values of $q>2$, since then Homma's lower bound $D(q) \ge A(q)/2$ is weaker. 
The following table provides for those small values of $q$ the best known lower bound for $A(q)/2$. For all other values of $q$, except possibly when $q$ is a prime, $A(q) \geq 2$.
\end{rem}

\begin{center}
\begin{tabular}{c| c |c }
$q$ & $A(q)/2 \geq$  & reference\\ \hline
$3$ & $0.2464$ &  \cite{DM13}\\
$4$ & $0.5$ &  \cite{Ihara,TVZ} \\
$5$ & $0.3636$ & \cite{Tem01,AM02} \\
$7$ & $0.4615$ & \cite{HS13}\\
$8$ & $0.75$ & \cite{Zink}\\
$11$ & $0.5714$ & \cite{HS13}\\
$13$ & $0.6$ & \cite{LM02}\\
$17$ & $0.8$ & \cite{LM02}\\
$19$ & $0.8$ & \cite{HS13}\\
$23$ & $0.9230$ & \cite{HS13}\\
$29$ & $0.9523$ & \cite{HS13}\\
$31$ & $0.9523$ & \cite{HS13}\\
\end{tabular}
\end{center}

\section*{Acknowledgements}

The first and second authors would like to acknowledge the support from The Danish Council for Independent Research (DFF-FNU) for the project \emph{Correcting on a Curve}, Grant No.~8021-00030B.

\end{document}